\documentclass[14pt]{amsart}
\usepackage{amssymb,amsthm,amsmath}
\usepackage{eucal}
\usepackage[numbers,sort&compress]{natbib}
\usepackage{color}
\usepackage{graphicx}
\usepackage{tikz}
\usepackage{amsmath,multirow}
\newcommand{\C}{{\mathbb C}}
\newcommand{\R}{{\mathbb R}}
\newcommand{\Z}{{\mathbb Z}}

\newcommand{\N}{{\mathbb N}}

\newcommand{\T}{{\mathbb T}}
\newcommand{\eps}{{\epsilon}}

\newcommand{\vp}{{\varphi}}
\newcommand{\hh}{
2\mathbf{i} \partial_{\bar z} h}

\title[ ]{Krylov--Bogolyubov averaging}

\author{Wenwen Jian}
\address{School of Mathematical Sciences, Fudan University, Shanghai 200433, P. R. China} \email{wwjian16@fudan.edu.cn}
\author{Sergei Kuksin}
\address{Universit\'e Paris-Diderot (Paris 7), UFR de Math\'ematiques - Batiment Sophie Germain, 5 rue Thomas Mann, 75205 Paris,
 France  \& School of Mathematics, Shandong University, Jinan, PRC \& Saint Petersburg State University, Universitetskaya nab., St. Petersburg, Russia}
\email{Sergei.Kuksin@imj-prg.fr}
\author{ Yuan Wu}
\address{School of Mathematical Sciences, Fudan University, Shanghai 200433, P. R. China} \email{14110840003@fudan.edu.cn}


\theoremstyle{plain}
\newtheorem{theorem}{Theorem}[section]
\newtheorem{corollary}[theorem]{Corollary}
\newtheorem{lemma}[theorem]{Lemma}
\newtheorem{proposition}[theorem]{Proposition}

\theoremstyle{definition}
\newtheorem{definition}[theorem]{Definition}

\newtheorem{example}[theorem]{Example}

\newcommand{\bigzero}{\mbox{\normalfont\Large\bfseries 0}}
\newcommand{\rvline}{\hspace*{-\arraycolsep}\vline\hspace*{-\arraycolsep}}

\begin{document}


\begin{abstract}
We present the modified approach to the classical Krylov-Bogolyubov averaging,\ developed recently for the purpose of PDEs. It allows to treat Lipschitz perturbations of linear systems  with pure imaginary spectrum and may be generalized to treat PDEs with small nonlinearities.
\end{abstract}
\maketitle

\numberwithin{equation}{section}
\section{Introduction}
The classical Krylov-Bogolyubov averaging method is a method for approximated analysis of nonlinear oscillating process.\
 Among a number of its equivalent or closely related formulations, we choose the following.\ In the space $ \mathbb{R}^{N}$,\ let us consider the differential equation of the form
\begin{eqnarray}\label{maineq}
\frac{\partial v}{\partial t}+ Av = \epsilon P(v),\ \ \ \ v(0) = v_{0},\ \ \ \ 0< \epsilon \le1,
\end{eqnarray}
where $ A$ is a linear operator with pure imaginary eigenvalues without Jordan cells,\ and $ P(v)$ is a  nonlinearity. The task is to study the behavior of solutions for (\ref{maineq}) on time-interval of order $ \epsilon^{-1}$. Let us firstly pass to the slow time $ \tau = \epsilon t$ and rewrite the equation as
\begin{eqnarray}\label{maineq1}
\frac{\partial v}{\partial \tau}+ \epsilon^{-1}Av = P(v),\ \ \ \  v(0) = v_{0},
\end{eqnarray}
where now $ \tau$ is a time-variable of order one.\ Secondly,\ in equation (\ref{maineq1}),\ let us pass to the interaction representation variables \footnote{under this name the change of variable (\ref{maineq2}) is known in  physics.}
 \begin{eqnarray}\label{maineq2}
a(\tau)= e^{\epsilon^{-1}\tau A}v(\tau),
\end{eqnarray}
and rewrite the equation as
 \begin{eqnarray}\label{maineq3}
\frac{\partial a}{\partial \tau} = e^{\epsilon^{-1}\tau A} P(e^{-\epsilon^{-1}\tau A}a), \ \ \ \ a(0) = v_{0}.
\end{eqnarray}
The Krylov-Bogolyubov averaging theorem is the following result:
\begin{theorem}\label{tbk}
Assume that the vector-field $P$  is locally Lipschitz continuous. Then

\noindent
1) the limit
\begin{eqnarray}\label{tbk1}
\langle\langle P\rangle\rangle(a) = \lim_{T\rightarrow  
\pm \infty}\frac{1}{|T|}\int_{0}^{T}e^{ s A} P(e^{- sA}a)\mathrm{d} s
\end{eqnarray}
exists for all $ a \in \mathbb{R}^{N}$ and also is Lipschitz continuous in $a$. 
\\
2) There exists $ \theta =\theta(|v_0|)
> 0$, such that for $ |\tau|\leq \theta $,\ a solution $ a^{\epsilon}(\tau)$ of equation (\ref{maineq3}) is  $ o(1)$-close,\ as $ \epsilon\rightarrow 0$,\ to a solution of the equation
$$
\frac{\partial a}{\partial \tau} = \langle\langle P\rangle\rangle(a), \ \ \ \ a(0) = v_{0}.
$$
\end{theorem}
We stress that the only restriction imposed on the spectrum of the operator $ A$ is that it is pure imaginary.\  Theorem~\ref{tbk} and related results were proved by Krylov-Bogolyubov in a number of works in 1930's.\ The research was summarized in the book \cite{BM},\ also see in \cite{AKN}.
In our work we present a proof of  Theorem~\ref{tbk}, based on a variation of the Krylov-Bogolyubov argument,
developed recently for the purposes of partial differential equations in \cite{H-G1, HKM}.\ It allows to prove the averaging theorem above under minimal restrictions on the smoothness of the nonlinearity $ P$ -- only its Lipschitz continuity is required -- and it generalizes to a class of perturbative problems in PDEs.
Theorem~\ref{tbk} is proved in Sections 3-4.\ In Section 5, we discuss its applications to the case when (\ref{maineq1}) is a Hamiltonian system.\ We remind that the Krylov-Bogolyubov averaging method was the first rigorously justified averaging theory. Before the work of Krylov-Bogolyubov the
 method of averaging existed as a heuristic theory, after that other rigorous averaging theories were created,\ see in \cite{AKN}. In particular, now the method of averaging applies
 to  equations with added stochasticity. The approach of our work, enriched with the  ideas of the seminal  work \cite{Khas}, suits well to the situation
 when the stochasticity is added to the problem in the form a stochastic force; both in the ODE and PDE settings. See the second half of the paper
 \cite{HKM} and references in that work.

Our paper is based on the lecture notes for a course that SK was teaching at a CIMPA School on Dynamical Systems in Kathmandu (October 25 -- November 5, 2018),\ which was attended by WJ and YW.
\smallskip

\noindent
{\bf Notation.} Abbreviation l.h.s. (r.h.s.) stands for ``left hand side" (``right hand side").  By $\mathbb{R}_+$ (by $\mathbb{Z}_+$)
we denote the set of non-negative real numbers (non-negative integers),
denote by $B_R$  the
open ball $ B_{R}= \{v : |v| < R\}$, $ R > 0$,\ and   by  $ \bar{B}_{R}$ -- its closure.
\smallskip

\noindent
{\bf Acknowledgements.}
SK  was supported   by the grant  18-11-00032 of Russian Science Foundation.\ WJ and YW were supported by the National Natural Science Foundation of China (No. 11790272 and No. 11421061).

\section{preliminaries}
Consider again equation (\ref{maineq1}),\ and assume that the linear operator $A$ has $N$ eigenvalues, counted with their geometrical
multiplicities. Assume also that these eigenvalues are pure imaginary. Then they go in pairs $\pm i\lambda_j$, where
 $ 0 \neq \lambda_{j} \in \mathbb{R}$  (see \cite{G}). So $ N$ is an even number, $N = 2n$.\ The imposed restrictions on $ A$ are equivalent to the following conditions (see \cite{G}):
$Ker A = \{0\}$ and in $ \mathbb{R}^{2n}$  exists a basis $\{e_{1}^{+},e_{1}^{-},e_{2}^{+},e_{2}^{-},...,e_{n}^{+},e_{n}^{-}\}$ such that in the corresponding coordinates $ \{x_{1},y_{1},x_{2},y_{2},...,x_{n},y_{n}\}$,\ the matrix of the linear operator $ A$
has the form
\begin{eqnarray}\label{m}  \begin{pmatrix}
\begin{matrix} 0 &-\lambda_{1}\\ \lambda_{1}&0 \end{matrix} & \rvline &
  \bigzero & \rvline & \dots & \rvline & \bigzero \\
\hline
 \bigzero & \rvline & \begin{matrix}0&-\lambda_{2}\\ \lambda_{2}& 0\end{matrix} & \rvline & \dots & \rvline & \bigzero\\
\hline
\vdots & \rvline & \vdots& \rvline & \ddots& \rvline& \vdots\\
\hline
\bigzero & \rvline & \bigzero & \rvline & \dots & \rvline & \begin{matrix} 0& -\lambda_{n}\\ \lambda_{n} & 0\end{matrix}
\end{pmatrix}.
\end{eqnarray}
Thus,\
the original unperturbed  linear system (\ref{maineq})$|_{\epsilon=0}$ reads:
\begin{eqnarray}\label{*}
\begin{cases}
\dot{x}_{1}- \lambda_{1}y_{1}=0,\\
\dot{y}_{1}+\lambda_{1}x_{1}=0,\\
\cdots\\
\dot{x}_{n}- \lambda_{n}y_{n}=0,\\
\dot{y}_{n}+\lambda_{n}x_{n}=0.
\end{cases}
\end{eqnarray}
Note that this linear system  can be written in the Hamiltonian form
$$
\dot{x} = -\frac{\partial h}{\partial y},\qquad
\dot{y} = \frac{\partial h}{\partial x},
$$
where $ h =-\frac{1}{2} \sum_{j=1}^{n}\lambda_{j}(x^{2}_{j}+y^{2}_{j})$.

\subsection{Complex structures in $ \mathbb{R}^{2n}$ and real analysis in $\mathbb{C}^{n} $. Systems  (\ref{*}) and (\ref{maineq1}) in complex notation.
}
The systems (\ref{*}) and (\ref{maineq1}) can be written more compactly if we introduce in the space $ \mathbb{R}^{2n}$ a complex structure and write $ A$ and the perturbation $P$ in its terms. Corresponding construction is performed  in this section and is used below to prove Theorem~\ref{tbk};  the complex
language allows to shorten the proof significantly.

Vectors in the  space $ \mathbb{R}^{2n}$ are caracterised by
 the coordinates $ (x_{1},y_{1},x_{2},y_{2},...,x_{n},y_{n}) $.\ Let us introduce in $ \mathbb{R}^{2n}$ a complex structure by denoting
\begin{eqnarray}\label{0}
\begin{cases}
{z}_{1} = {x}_{1} + \mathbf{i}{y}_{1},\\
{z}_{2} = {x}_{2} + \mathbf{i}{y}_{2},\\
\dots\\
{z}_{n} = {x}_{n} + \mathbf{i}{y}_{n}.
\end{cases}
\end{eqnarray}
Then the real space $ \mathbb{R}^{2n}$ becomes a space of complex sequences $ z= (z_{1},z_{2},...,z_{n})$ with $ z_{j}\in \mathbb{C}$.\ That is,\ we have achieved that
\begin{eqnarray}\nonumber
\mathbb{R}^{2n}\simeq \mathbb{C}^{n}.
\end{eqnarray}
In the complex notation,\ the Euclidean scalar product $ \langle\cdot,\cdot\rangle$ in $ \mathbb{R}^{2n}\simeq \mathbb{C}^{n}$  reads
\begin{eqnarray}\label{0*}
\langle z,z'\rangle = \Re (\sum_{j} z_{j} \bar{z}'_{j})=: \Re(z\cdot \bar{z}').
\end{eqnarray}
For the real numbers $ \lambda_{1},\lambda_{2},...,\lambda_{n}$ as in \eqref{*}
 let us consider the linear operator
\begin{equation*}
 \text{diag}\{\mathbf{i}\lambda_{j}\}: \ \mathbb{C}^{n}\rightarrow \mathbb{C}^{n},
 \ \ (z_{1},z_{2},...,z_{n})\mapsto (\mathbf{i}\lambda_{1}z_{1},\mathbf{i}\lambda_{2}z_{2},...,\mathbf{i}\lambda_{n}z_{n}).
\end{equation*}
In the real coordinates $ (x_{1},y_{1},x_{2},y_{2},...,x_{n},y_{n})$  it  reads
\[
(x_{1},y_{1},x_{2},y_{2},...,x_{n},y_{n})\mapsto (-\lambda_{1}y_{1},\lambda_{1}x_{1}, -\lambda_{2}y_{2},\lambda_{2}x_{2},...,-\lambda_{n}y_{n},\lambda_{n}x_{n}).
\]
That is,\ in the complex coordinates the operator $ A$ with the matrix (\ref{m}) is the operator $\text{diag}\{\mathbf{i}\lambda_{j}\} $,\ so the system of linear equations (\ref{maineq})$|_{\epsilon=0}=$\,\eqref{*}
reduces to the  diagonal complex system
$$
\dot{v}_{j}+ \mathbf{i}\lambda_{j} v_{j} =0,\ \  1\leq j\leq n.
$$

 In the complex notation the perturbed  system (\ref{maineq}) reads
\begin{eqnarray}\label{2}
\dot{v}_{j}+ \mathbf{i}\lambda_{j} v_{j} = \epsilon P_{j}(v),\ \ v(0)=v_0,\ \    v= (v_{1},v_{2},...,v_{n})\in \mathbb{C}^{n}.
\end{eqnarray}
Below we assume that the vector-field $ P$ is locally  Lipschitz, i.e. its restrictions to bounded balls $ B_{R},\ R > 0 $, are Lipschitz-continuous. The case of polynomial vector-field $ P$ will be for us of special interest,\ and we start with its brief discussion.
\begin{definition}
A complex function
$ F : \mathbb{C}^{n}\rightarrow \mathbb{C}$ is a polynomial if it can be written as
$$
 F(z) = \underset{0\leq |\alpha|,|\beta|\leq {M}}{\sum} C_{\alpha\beta}z^{\alpha}\bar{z}^{\beta},
$$
where $ \alpha=(\alpha_{1},\alpha_{2},\dots, \alpha_{n}), \beta=(\beta_{1},\beta_{2},\dots, \beta_{n}) \in \mathbb{Z}_{+}^{n}$
 are multi-indices with the norm $ |\alpha|= |\alpha_{1}|+|\alpha_{2}|+\cdot\cdot\cdot+|\alpha_{n}|$,\ $ C_{\alpha\beta}$ are some complex numbers,\ and
\[
z^{\alpha} = \prod^{n}_{j=1}z_{j}^{\alpha_{j}},\qquad \bar{z}^{\beta}= \prod^{n}_{j=1}\bar{z}_{j}^{\beta_{j}}.
\]
\end{definition}\label{5*}
\begin{definition}
A vector-field $ P(z)$ is polynomial if every its component $ P_{j}$ is a polynomial function.
\end{definition}
We recall that for a function $ f(z)$ (real or complex) of a complex variable $z= x+ \mathbf{i}y$,
the derivatives ${\partial f}/{\partial z}$ and ${\partial f} /{\partial \bar z}$ are defined as
$
\frac{\partial f}{\partial z}= \frac{1}{2}(\frac{\partial f}{\partial x}-\mathbf{i}\frac{\partial f}{\partial y})$ and
$\frac{\partial f}{\partial \bar{z}}= \frac{1}{2}(\frac{\partial f}{\partial x}+\mathbf{i}\frac{\partial f}{\partial y}).
$
The   lemma  below follows by elementary calculation:
\begin{lemma}\label{b1}
Let $ z\in\mathbb{C}$ and consider a complex polynomial $ F(z) = \underset{m,n=1}{\overset{M}{\sum}} C_{mn}z^{m}\bar{z}^{n}$.\ Then
$
\frac{\partial F}{\partial z} = \underset{m,n=1}{\overset{M}{\sum}} mC_{mn}z^{m-1}\bar{z}^{n},
$ and $
\frac{\partial F}{\partial \bar{z}} = \underset{m,n=1}{\overset{M}{\sum}} nC_{mn}z^{m}\bar{z}^{n-1}.
$
 If $ F : \mathbb{C}^{n}\rightarrow \mathbb{R}$ is a real-valued $ C^{1}$-smooth function,\ then
$ \overline{\frac{\partial F}{\partial z_{j}}} = \frac{\partial F}{\partial \bar{z}_{j}}$,
for any $ 1\leq j\leq n $.
\end{lemma}

\begin{example}\label{ex1}
For $n=1$,\ consider the system of equations
\begin{equation}\label{ss1}
\dot x -\omega y =2\eps x(x^2-y^2),\quad \dot y +\omega x=4\eps yx^2.
\end{equation}
Now ${\color{red}\lambda}=\omega$ and in the complex coordinate $v=x+iy$ the system reads
\begin{equation}\label{s1}
\dot v+i  \lambda v= \eps( v^2\bar v + v^3).
\end{equation}
\end{example}
\medskip

Let us come back to the general case of locally Lipschitz vector-fields $ P$.

\begin{definition}
\label{8} Let $ \mathcal{X} : \mathbb{R}_{+} \rightarrow \mathbb{R}_{+}$ be a non-decreasing continuous function and $ f : \mathbb{C}^{n}\rightarrow \mathbb{C}^{n}$ be a continuous vector-field.\ We say that $ f \in \text{Lip}_{\mathcal{X}}(\mathbb{C}^{n},\mathbb{C}^{n})$,\ if for any $R \geq 0 $,\ $|f|_{B_{R}}\leq \mathcal{X}(R)$ and
$ \text{Lip} f|_{B_{R}} \leq \mathcal{X}(R)$.
\end{definition}
\begin{example}
 Let $ P : \mathbb{C}^{n}\rightarrow \mathbb{C}^{n}$ be a $ C^{1}$-smooth vector-field. For
 $ v\in \mathbb{C}^{n} $\ we denote by $ dP(v)$ the  differential of $ P$ at $ v$ (this is a linear  over real numbers map
 from $ \mathbb{C}^{n}$ to $ \mathbb{C}^{n}$).  Denote
 $ \widetilde{\mathcal{X}}(R;P)= \max\left\{\sup_{B_{R}}\|dP(v)\|, \sup_{B_{R}}|P(v)|\right\}$. Then $ \widetilde{\mathcal{X}}$ defines
  a continuous function of $ R \ge 0$,
  and $  P \in \text{Lip}_{\widetilde{\mathcal{X}}}(\mathbb{C}^{n},\mathbb{C}^{n})$.\ Indeed,\ the continuity of $ \widetilde{\mathcal{X}}$
   is obvious,\ while the second property follows from the mean-value theorem which implies that
$$ |P(u_2)- P(u_1)| \leq \widetilde{\mathcal{X}}(R;P)|u_{2}-u_{1}|\ \ \  \text{if} \ \ u_{1}, u_{2} \in B_{R}.$$
\end{example}

\begin{lemma}\label{a2}
Let $ P \in \text{Lip}_{\mathcal{X}}(\mathbb{C}^{n},\mathbb{C}^{n})$ and $ v_0 \in \bar{B}_{R}$.\ Denote $\theta= \frac{R}{\mathcal{X}(2R)}$.\ Then a solution $ v(t)$ of (\ref{2}) exists for  $|t|\leq \epsilon^{-1}\theta $ and stays in the ball $ \bar{B}_{2R}$.
\end{lemma}
\begin{proof}
Since $ P$ is a  locally Lipschitz vector-field,\ then a solution $ v(t)$ of equation (\ref{2}) exists till the blow-up time.
Taking the scalar product of equation (\ref{2}) with $ v(t)$ (see (\ref{0*})),\ we get
\[
\frac{1}{2}\frac{\text{d}}{\text{d}t}|v(t)|^{2} = - \mathbf{i} \langle \text{diag}(\lambda_{j})v,v\rangle + \epsilon \langle P(v),v\rangle =\epsilon \langle P(v),v\rangle.
\]
Denote $ T = \inf\left\{t \in [0,\epsilon^{-1}\theta]: |v(t)|\geq 2R\right\}$,\ where $ T$ equals $\epsilon^{-1}\theta$ if the set under the inf-sign is empty.\ Then for $ 0 <t \leq T$ we have
\[
\frac{1}{2}\frac{\text{d}}{\text{d}t}|v(t)|^{2}\leq \epsilon |v||P(v)|\leq \epsilon \mathcal{X}(2R)|v|.
\]
Thus,\ $ \frac{\text{d}}{\text{d}t}|v(t)| \leq \epsilon \mathcal{X}(2R)$ and $ |v(t)| \leq R + \epsilon \mathcal{X}(2R)t < 2R$
for all $ 0 <t <\epsilon^{-1} \theta$.\ So $ T= \epsilon^{-1}\theta $ and the result follows.
\end{proof}

\subsection{Slow time and interaction representation}
Denote  $ \tau = \epsilon t$.\ Then
$
\frac{\partial v}{\partial t}= 
   \epsilon \frac{\partial v}{\partial \tau}
$,\
so equations (\ref{2}) reduce to
\begin{eqnarray}\label{10}
\frac{\partial{v}_{j}}{\partial \tau}+ \mathbf{i}\epsilon^{-1}\lambda_{j} v_{j} = P_{j}(v),\ \ \ \  1\leq j \leq n.
\end{eqnarray}
Let us substitute in (\ref{10})
\begin{eqnarray}\label{11}
{v}_{j}(\tau) = e^{- \mathbf{i}\epsilon^{-1}\lambda_{j}\tau}a_{j}(\tau),\ \ \ \  1\leq j \leq n.
\end{eqnarray}
Then (\ref{10}) becomes
\begin{eqnarray}\label{12}
\dot{{a}}_{j} = e^{ \mathbf{i}\epsilon^{-1}\lambda_{j}\tau}P_{j}(v),\ \ \ \  1\leq j \leq n.
\end{eqnarray}
Denote $ a(\tau)= (a_{1}(\tau),a_{2}(\tau),...,a_{n}(\tau))\in  \mathbb{C}^{n}$.\
For a real vector $ w = (w_{1},w_{2},...,w_{n}) \in\mathbb{R}^{n}$ let $\Phi_{w}$ be the rotation operator
$$
\label{13} \Phi_{w}: \mathbb{C}^{n}\rightarrow \mathbb{C}^{n},\ \ \ \ \Phi_{w}= \text{diag}\{ e^{\mathbf{i}w_{1}},e^{\mathbf{i}w_{2}},...,e^{\mathbf{i}w_{n}}\}.
$$
It is easy to see that
\[
(\Phi_{w})^{-1}= \Phi_{-w},\quad  \Phi_{w_{1}}\circ\Phi_{w_{2}}= \Phi_{w_{1}+w_{2}}, \quad \Phi_0= \text{id},
\]
and that each $\Phi_{w}$ is a unitary transformation.\\
Denote by $ \Lambda $ the vector $ (\lambda_{1},\lambda_{2},...,\lambda_{n}) \in \mathbb{R}^{n}$.\ Then (\ref{11}) can be written as
$
{v}(\tau) = \Phi_{- \tau\epsilon^{-1}\Lambda}a(\tau),
$
or
$
{a}(\tau) = \Phi_{\tau\epsilon^{-1}\Lambda}v(\tau).
$
Thus,\ the system (\ref{12}) reads as
\begin{eqnarray}\label{maineq**}
\frac{\partial a}{\partial \tau} = \Phi_{\tau\epsilon^{-1}\Lambda}\circ P(\Phi_{- \tau\epsilon^{-1}\Lambda}a(\tau)),\ \ \ \   |\tau|\leq \theta,
\end{eqnarray}
with the initial condition
\begin{eqnarray}
\label{17} a(0)= v_0,\ \ \ \  |v_0|=: R.
\end{eqnarray}
Note that
\begin{eqnarray}
\label{18} |a_{j}(\tau)| = |v_{j}(\tau)|,\ \ \ \   \forall\tau,\ \ \ \ 1\leq j \leq n.
\end{eqnarray}

\section{Averaging of vector-fields }\label{18*} We recall that a diffeomorphism $ G: \mathbb{R}^{2n}\rightarrow \mathbb{R}^{2n}$ transforms a vector-field $ W$ on $ \mathbb{R}^{2n}$ to the vector-field
$
(G_{*}W)(v) = dG(u)(W(u))$, $ u= G^{-1}(v)
$,\
see in \cite{Arn}.\ Accordingly,\ a linear isomorphism $ \Phi_{\Lambda t},\ t\in \R$,\ transforms the vector-field $ P$ to
$$
((\Phi_{\Lambda t})_{*}P)(v)= \Phi_{\Lambda t}\circ P(\Phi_{-\Lambda t}  v).
$$
Our goal in this section is to study the averaging in $ t$ of the vector-field above.

For a continuous vector-field $ P $ on $\mathbb{C}^{n}$ and a vector $ \Lambda\in (\mathbb{R}\setminus \{0\})^{n}$,\ we denote
\begin{eqnarray}\label{19}
\langle\langle P\rangle\rangle(a) = \lim_{T\rightarrow \pm\infty}\frac{1}{|T|}\int_{0}^{T}\Phi_{\Lambda t}\circ P(\Phi_{-\Lambda t}a)\mathrm{d} t,
\end{eqnarray}
if the limit exists (for $T<0$ we understand $\int_0^T\dots \text{d}t$ as the integral $\int_T^0\dots \text{d}t$).
Denote
\begin{eqnarray}
\label{20}
y^{t}(a) = \Phi_{\Lambda t}\circ P(\Phi_{-\Lambda t}a),\ \
\end{eqnarray}
and for $T\ne0$ set
\begin{eqnarray}
\label{21} \langle\langle P\rangle\rangle^{T}(a) = \frac{1}{|T|}\int_{0}^{T}y^{t}(a) \mathrm{d} t.
\end{eqnarray}
Then
$$
\langle\langle P\rangle\rangle(a) = \lim_{T\rightarrow \pm\infty} \langle\langle P\rangle\rangle^{T}(a).
$$
The vector-field  $ \langle\langle P\rangle\rangle$ is called the averaging of a field  $ P$ in the direction of a vector $ \Lambda$,\ and $ \langle\langle P\rangle\rangle^{T}(a)$ is called the partial  averaging.\ The latter always exists.\ Our goal in this section is to prove that the former also exists,\ if the mapping $ P$ is locally Lipschitz.\ To indicate the dependence of the two introduced objects on $ \Lambda$,\ sometimes we will write them as $\langle\langle\cdot\rangle\rangle_{\Lambda}$ and $ \langle\langle\cdot\rangle\rangle^{T}_{\Lambda}$.\ We recall that being written in the special basis the matrix of operator $ A$ takes the form (\ref{m}), and that after introducing in $ \mathbb{R}^{2n}$ the complex structure (\ref{0}) the matrix becomes $ \text{diag}\{\mathbf{i}\lambda_{j}\}$.\ Since $ \Phi_{\Lambda t}= \exp(\text{diag}\{\mathbf{i}\lambda_{j}\}t)$,\ then the definition of $ \langle\langle P\rangle\rangle$ agrees with that in (\ref{tbk1}).

\begin{lemma}\label{a3}
If $ P \in \text{Lip}_{\mathcal{X}}(\mathbb{C}^{n},\mathbb{C}^{n})$,\ then
$ \langle\langle P\rangle\rangle^{T} \in \text{Lip}_{\mathcal{X}}(\mathbb{C}^{n},\mathbb{C}^{n})$ for any $ T\ne0$.
\end{lemma}
\begin{proof}
If $ a \in \bar{B}_{R}$,\ then $ \Phi_{-\Lambda t}a \in \bar{B}_{R}$.\ Since $ P \in \text{Lip}_{\mathcal{X}}(\mathbb{C}^{n},\mathbb{C}^{n})$,\ one obtains
$
|P(\Phi_{-\Lambda t}a)| \leq \mathcal{X}(R)
$,\ for each $ t$,\
and then
\begin{eqnarray}\label{22}
| y^{t}(a)| \leq \mathcal{X}(R),\ \ \ \  \forall t.
\end{eqnarray}
Similarly,\ for any $ a_{1}, a_{2} \in \bar{B}_{R} $,\ we have
\begin{eqnarray}
\label{23} | y^{t}(a_{2})- y^{t}(a_{1})| &=& |P(\Phi_{-\Lambda t}a_{2})-P(\Phi_{-\Lambda t}a_{1})|\\
\nonumber &\leq& \mathcal{X}(R)|\Phi_{-\Lambda t}a_{2} - \Phi_{-\Lambda t}a_{1}|\\
\nonumber &\leq& \mathcal{X}(R)|a_{2} - a_{1}|,\ \ \ \   \forall t.
\end{eqnarray}
From (\ref{22}) and (\ref{23}),\ one obtains
\begin{eqnarray}\nonumber
| \langle\langle P\rangle\rangle^{T}(a)| \leq \sup_{| t| \leq |T|}| y^{t}(a)|\leq \mathcal{X}(R),
\end{eqnarray}
and
\begin{eqnarray}
\nonumber  &&| \langle\langle P\rangle\rangle^{T}(a_{2})- \langle\langle P\rangle\rangle^{T}(a_{1})|
\leq\sup_{| t| \leq |T|}| y^{t}(a_{2})- y^{t}(a_{1})|
\leq \mathcal{X}(R)|a_{2} - a_{1}|.
\end{eqnarray}
Thus,\ $\langle\langle P\rangle\rangle^{T} \in \text{Lip}_{\mathcal{X}}(\mathbb{C}^{n},\mathbb{C}^{n})$ for  $ T\ne0$.
\end{proof}

\begin{lemma}[The main lemma of averaging]\label{a4}
For any $ \Lambda \in (\mathbb{R}\setminus \{0\})^{n}$ and $ P \in \text{Lip}_{\mathcal{X}}(\mathbb{C}^{n},\mathbb{C}^{n})$,\ the limit of (\ref{19}) exists for any $a \in \mathbb{C}^{n}$,\ and  $ \langle\langle P\rangle\rangle \in \text{Lip}_{\mathcal{X}}(\mathbb{C}^{n},\mathbb{C}^{n})$.\ If $ a \in \bar{B}_{R}$,\ then the rate of convergence in  (\ref{19}) depends only on $ R, \Lambda$ and $ P$.
\end{lemma}
Before proving the general case of Lemma \ref{a4},\ we firstly consider the  case when $ P$ is a polynomial vector-field.\ Then,\
\begin{eqnarray}\label{23*}
P_{j}(v)= \sum_{0 \leq |\alpha|,|\beta|\leq N}C^{\alpha\beta}_{j}v^{\alpha}\bar{v}^{\beta},\ \  1 \leq j\leq n,
\end{eqnarray}
and one has
\begin{eqnarray}
\nonumber y^{t}_{j}(v) &=& e^{\mathbf{i}\lambda_{j}t}\sum_{0 \leq |\alpha|,|\beta|\leq N}C^{\alpha\beta}_{j}\prod_{i}(e^{-\mathbf{i}\lambda_{i}t}v_{i})^{\alpha_{i}}
\prod_{i}(e^{\mathbf{i}\lambda_{i}t}\bar{v}_{i})^{\beta_{i}}\\
\nonumber &=& \sum_{0 \leq |\alpha|,|\beta|\leq N}C^{\alpha\beta}_{j}
e^{\mathbf{i}t(\lambda_{j}-\Lambda\cdot \alpha + \Lambda\cdot \beta)}v^{\alpha} \bar{v}^{\beta}.
\end{eqnarray}
It follows that
\begin{eqnarray}\nonumber
 \langle\langle P\rangle\rangle^{T}_j(v) = \sum_{0 \leq |\alpha|,|\beta|\leq N}C^{\alpha\beta}_{j}\left(\frac{1}{|T|}\int_{0}^{T} e^{\mathbf{i}t(\lambda_{j}-\Lambda\cdot \alpha + \Lambda\cdot \beta)}\mathrm{d} t \right)v^{\alpha} \bar{v}^{\beta}.
\end{eqnarray}

\begin{definition}
A pair $ (\alpha, \beta)$ is called $ (\Lambda,j)$-resonant if $ \lambda_{j}-\Lambda\cdot \alpha + \Lambda\cdot \beta = 0$. The resonant part
of the polynomial vector-field \eqref{23*} is another polynomial vector-field $P^{\text{res}}(v)$ such that
$$
P_j^{\text{res}}(v)
= \sum_{\substack{\text{pair} \ (\alpha,\beta)\  \text{is}\
(\Lambda,j)\text{-resonant},\\ 0 \leq |\alpha|,|\beta|\leq N}}C^{\alpha\beta}_{j}v^{\alpha} \bar{v}^{\beta}\quad\text{for}\quad 1\le j\le n.
$$
\end{definition}

\begin{example}\label{ex2}
For  equation \eqref{s1} in Example \ref{ex1},\ we have
$
P(v) = v^2\bar v +v^3 =: P_1(v) +P_2(v).
$
The monomial $P_1$ is resonant since
$
\omega -\omega \cdot2 +\omega=0.
$
So now $P^{\text{res}}(v) = P_1(v)$.
\end{example}
\medskip

Note that
\begin{eqnarray*}
&&\frac{1}{|T|} \int_{0}^{T} e^{-\mathbf{i}t(\lambda_{j}-\Lambda\cdot \alpha + \Lambda\cdot \beta)}\mathrm{d} t\\
\nonumber&&\qquad= \begin{cases} 1, &\ \ \ \text{if}\ (\alpha,\beta)\ \text{is}\ (\Lambda,j)\text{-resonant},\\
\frac{1}{-\mathbf{i}T(\lambda_{j}-\Lambda\cdot \alpha + \Lambda\cdot \beta)}(e^{-\mathbf{i}T(\lambda_{j}-\Lambda\cdot \alpha + \Lambda\cdot \beta)}-1),  &\ \ \ \text{otherwise}.
\end{cases}
\end{eqnarray*}
So,\
\begin{eqnarray}\label{25}
\lim_{T\rightarrow \pm\infty} \frac{1}{|T|}\int_{0}^{T} e^{-\mathbf{i}t(\lambda_{j}-\Lambda\cdot \alpha + \Lambda\cdot \beta)}\mathrm{d} t
= \begin{cases} 1, &\ \ \ \ \ \text{if}\ (\alpha,\beta)\ \text{is}\ (\Lambda,j)\text{-resonant},\\
0,  &\ \ \ \ \ \text{otherwise}.
\end{cases}
\end{eqnarray}
Thus, we have
\begin{eqnarray}\label{25*}
 \langle\langle P\rangle\rangle^{T}_{j}(v) \xrightarrow[T\rightarrow \pm \infty]{}
   P^{\text{res}}_{j}(v).
\end{eqnarray}
Therefore,\ in the polynomial case  the limit in (\ref{19}) exists.
\begin{lemma}\label{a5}
If $ P \in \text{Lip}_{\mathcal{X}}(\mathbb{C}^{n},\mathbb{C}^{n})$ is a polynomial vector-field of the form (\ref{23*}),
then the limit $ \langle\langle P\rangle\rangle $ in
 (\ref{19}) exists for all $ a\in\mathbb{C}^n$, equals to the resonant part $P^{res}$ of $ P$
and satisfies  $ \langle\langle P\rangle\rangle  \in \text{Lip}_{\mathcal{X}}(\mathbb{C}^{n},\mathbb{C}^{n})$.\
 Moreover,\ if $ a \in \bar B_{R}$,\ then the rate of convergence (\ref{19}) depends only on $ R, \Lambda$ and $ P$.
\end{lemma}
\begin{proof}
The existence of the limit $ \langle\langle P\rangle\rangle $ already  is  proved, and its Lischitz continuity easily follows Lemma \ref{a3}.\
The second assertion holds since the rate of convergence in (\ref{25}) and (\ref{25*}) depends only on the indicated quantities.\ We omit the details.
\end{proof}

Now we begin to prove  Lemma \ref{a4}.
\begin{proof}
To show that the limit exists we have to verify that $ \frac{1}{|T|}\int_{0}^{T}e^{\textbf{i}\lambda_j t}\circ P_j(\Phi_{-\Lambda t}a)\mathrm{d} t$
converges to a limit as $T \rightarrow \pm\infty$.\ It suffices to show that for any $\xi>0$, there exists $T'=T'(\xi,{R},\Lambda,P)>0$,\ such that
\begin{equation}\label{minus}
  |\langle\langle P \rangle\rangle^{ T_1}_{j}(a)-\langle\langle P \rangle\rangle^{T_2}_{j}(a)|\leq \xi,\ \ \ \ \text{if }\  |T_1|, |T_2| \geq T'.
\end{equation}
For any $j$
consider the restriction of
 $P_j$ to the  closed ball $\bar{B}_{R}$. By the Stone-Weierstrass theorem, there exist $N$ and a polynomial $P_{j}^{N}(a)$ of degree $N$,\ depending only on $R$ and $ P$,\ such that
\begin{equation*}
  |P_j(a)-P_j^{N}(a)|\leq \frac{\xi}{4},\ \ \  \forall a\in\bar{B}_{R}.
\end{equation*}
We have got a polynomial vector field $P^N$, for which the  assertions of  Lemma \ref{a4} already are proved.

Since $\Phi_{-\Lambda t}a\in \bar{B}_{R}$ for any $t$, then
\begin{equation*}
  |y_j^{t}(a;P_j)-y_j^{t}(a;P_j^{N})|\leq \frac{\xi}{4},\ \ \ \forall t,\ \forall a\in \bar{B}_{R}.
\end{equation*}
So,
\begin{equation}\label{ppn}
  |\langle\langle P \rangle\rangle_{j}^{T}(a)-\langle\langle P^N \rangle\rangle_{j}^{T}(a)|\leq \frac{\xi}{4},\ \ \ \ \forall T\ne0.
\end{equation}
By Lemma \ref{a5}, there exists $T'=T'_N>0$ such that
\begin{equation}\label{pnt}
  |\langle\langle P^{N} \rangle\rangle_{j}^{\tilde{T}}(a)-\langle\langle P^N \rangle\rangle_{j}(a)|\leq \frac{\xi}{4},\ \ \ \ \forall |\tilde{T} |\geq T'.
\end{equation}
From (\ref{ppn}) and (\ref{pnt}),
\begin{equation*}
  |\langle\langle P \rangle\rangle_{j}^{T_1}(a)-\langle\langle P^N \rangle\rangle_{j}(a)|\leq \frac{\xi}{2},\ \ \ \ \forall |{T}_1|\geq T'.
\end{equation*}
The same is true for $T_2$. Therefore (\ref{minus}) follows, and the convergence (\ref{19}) is established. The inclusion
 $\langle\langle P\rangle\rangle\in \text{Lip}_{\mathcal{X}}(\C^n, \C^n)$  is a consequence of Lemma \ref{a3}, while the
   last assertion of the lemma directly follows from the proof. The lemma is proved.
\end{proof}

\begin{example}
Let $\lambda_j>0$ for all $j$ and $P_j$ be an anti-holomorphic polynomial $P_j(v)=\sum_{|\beta|\leq N}C_j^{\beta}\bar{v}^{\beta}$.
Then no pair $(0,\beta)$ is $(\Lambda,j)$-resonant and  $\langle\langle P\rangle\rangle=0$.
\end{example}

\subsection{Properties of the operator  $\langle\langle\cdot\rangle\rangle$}
\begin{proposition}\label{p}
Let $P$ and $Q$ be locally Lipschitz vector-fields on $\mathbb{C}^n$. Then
\begin{itemize}
  \item[1)] (linearity): $\langle\langle aP+bQ \rangle\rangle=a\langle\langle P\rangle\rangle+b\langle\langle Q \rangle\rangle$ for any $a,b\in\mathbb{R}$.
  \item[2)] If  $P=\text{diag}(a_1,\dots,a_n)$, $a_j\in\mathbb{C}$, then $\langle\langle P\rangle\rangle=P$.
  \item[3)] The mapping $v\mapsto\langle\langle P\rangle\rangle(v)$ commutes with all operators $\Phi_{\Lambda\theta},\ \theta\in\mathbb{R}$.
  \item[4)] The mapping $(\mathbb{C}^n\times(\mathbb{R}\setminus \{0\})^n)\ni (v,\Lambda)\mapsto \langle\langle P\rangle\rangle_{\Lambda}(v)$ is
  measurable. So for every $v$ the averaging $\langle\langle P\rangle\rangle_{\Lambda}(v)$ is a measurable function of $\Lambda$.
\end{itemize}
\end{proposition}

\begin{proof}
Properties 1) and 2) are obvious. Let us prove 3), assuming for definiteness that $T>0$.
We have:
\begin{align*}
  \langle\langle P\rangle\rangle(\Phi_{\Lambda\theta}a)&= \lim_{T\rightarrow + \infty}\frac{1}{T}\int_0^T \Phi_{\Lambda t}\circ P(\Phi_{-\Lambda t}\Phi_{\Lambda\theta}a)\text{d}t \\
 &=\Phi_{\Lambda\theta}\lim_{T\rightarrow+ \infty}\frac{1}{T}\int_0^T \Phi_{\Lambda (t-\theta)}\circ P(\Phi_{-\Lambda (t-\theta)}a) \text{d}t\\
 &=\Phi_{\Lambda\theta}\lim_{T\rightarrow+\infty}\frac{1}{T}\int_{-\theta}^{T-\theta} \Phi_{\Lambda t'}\circ P(\Phi_{-\Lambda t'}a) \text{d}t'\\
 &=\Phi_{\Lambda\theta}\lim_{T\rightarrow+\infty}\frac{1}{T}\left(\int_0^{T}+\int_{-\theta}^0-\int_{T-\theta}^T \right)\Phi_{\Lambda t'}\circ P(\Phi_{-\Lambda t'}a) \text{d}t'\\
 &=\Phi_{\Lambda\theta}\langle\langle P\rangle\rangle(a).
\end{align*}

To prove 4), we note that for $T>0$ the mappings $(a,\Lambda)\mapsto \langle\langle P\rangle\rangle_{\Lambda}^{T}(a)$ are continuous, so measurable. By Lemma \ref{a4}, the mapping in question is a point-wise limit, as $T\rightarrow \infty$, of the measurable mappings above; so it also is measurable (see \cite{R}, Theorem~1.14).
\end{proof}

\section{Averaging for solutions of  equation (\ref{maineq**})} \label{s_avr}
In this section we get our main result, describing the behaviour, as $\epsilon\rightarrow 0$, of solutions of equation  (\ref{maineq**}) on time-intervals $|\tau|\leq \text{const}$, where {\it const} does not depend on $\epsilon$. In view of (\ref{maineq**}), this also describes the
behaviour of the amplitudes $|v_j(\tau)|$ of solutions  (\ref{10}), and accordingly, the behaviour of the amplitudes of solutions for (\ref{2}) on long time-intervals $|t|\leq \text{const}\  \epsilon^{-1}$.

Let in eq. (\ref{10}) $ P \in \text{Lip}_{\mathcal{X}}(\mathbb{C}^{n},\mathbb{C}^{n})$ for some function $\mathcal{X}$ as in Definition \ref{8}, and let $v(\tau)$ be its solution such that  $v(0)=v_0$. Denote $|v_0|=R$.
Then by Lemma \ref{a2}, $v(\tau)\in\bar{B}_{2R}$ for $|\tau|\leq \theta=\frac{R}{\mathcal{X}(2R)}$.  For $|\tau|\leq \theta$ the curve
$a^{\epsilon}(\tau)=\Phi_{\tau\epsilon^{-1}\Lambda}v(\tau)$ satisfies \eqref{maineq**},  \eqref{17} and
 $ |a^{\epsilon}_j(\tau)| = |v_j(\tau)|$  for each $j$. So for $ |\tau|\leq \theta$ we have:
\begin{equation}\label{ae}
  \left\{
     \begin{array}{c}
      |a^{\epsilon}(\tau)|\leq 2R,\ \ \ \   \\
          |\frac{\partial a^{\epsilon}}{\partial \tau}| \leq |P(\Phi_{-\tau\epsilon^{-1}\Lambda}a(\tau))|\leq \mathcal{X}(2R).
     \end{array}
   \right.
\end{equation}
Consider the collection of curves $a^{\epsilon}$ (the solutions of equation (\ref{maineq**})),
\begin{equation*}
  a^{\epsilon}\in C([-\theta,\theta],\mathbb{C}^n), \quad \epsilon \in (0,1].
\end{equation*}
By (\ref{ae}) and the Arzela-Ascoli theorem, the family $\{a^{\epsilon}, 0<\epsilon\leq 1\}$ is precompact in $C([-\theta,\theta],\mathbb{C}^n)$. So there exists a sequence
$ \epsilon_j\rightarrow 0$, such that
\begin{equation}\label{compact}
  a^{\epsilon_j}\rightarrow a^0 \text{ in }C([-\theta,\theta],\mathbb{C}^n),\ \ \ \ \text{ as }\epsilon_j\rightarrow 0,
\end{equation}
for some curve $a^0\in C([-\theta,\theta],\mathbb{C}^n)$.
By (\ref{ae}),
\begin{equation*}
  |a^{\epsilon}(\tau_1)-a^{\epsilon}(\tau_2)|\leq \mathcal{X}(2R)|\tau_1-\tau_2|.
\end{equation*}
Passing in this relation to the limit as $\epsilon_j\rightarrow 0$, we obtain
\begin{equation}\label{a0}
  |a^{0}(\tau_1)-a^{0}(\tau_2)|\leq \mathcal{X}(2R)|\tau_1-\tau_2|,\ \ \ \  \forall \tau_1,\tau_2\in[-\theta,\theta].
\end{equation}

Now we address the following problem: does the limit $a^0$ depend on $\{\epsilon_j\}$? If it does not, then how to describe it?

A solution $a^{\epsilon}(\tau)$ of (\ref{maineq**}) satisfies the relation
\begin{equation}\label{aet}
  a^{\epsilon}(\tau)=v_0+\int_0^{\tau}\Phi_{s\epsilon^{-1}\Lambda}\circ P(\Phi_{-s\epsilon^{-1}\Lambda}a^{\epsilon}(s))\text{d}s,\ \ \ \ \forall |\tau|\leq \theta,
\end{equation}
and the estimates \eqref{ae}. From Lemma \ref{a4},
\begin{equation}\label{fromlemma}
  \frac{1}{|T|}\int_0^T \Phi_{t\epsilon^{-1}\Lambda}\circ P(\Phi_{-t\epsilon^{-1}\Lambda}a^{\epsilon}(t))\text{d}t=
  \langle\langle P\rangle\rangle(a)+o(1),\ \ \text{as } T\rightarrow \pm\infty,
\end{equation}
where $o(1)$ does not depend on $v_0$ if $v_0\in \bar B_R$.

Consider the following  effective equation
\begin{equation}\label{effective}
   a(\tau)=v_0+\int_0^{\tau}\langle\langle P\rangle\rangle(a(s))\text{d}s,
\end{equation}
that is
$$
\frac{\partial}{\partial \tau}
{a}(\tau)=\langle\langle P\rangle\rangle(a(\tau)),\qquad a(0)=v_0.
 $$
Since $\langle\langle P\rangle\rangle$ is locally Lipschitz, then a solution for (\ref{effective}) is unique and exists at least for small $|\tau|$.

\begin{lemma}\label{solutiona0}
The curve
$a^0(\tau)$ is a solution of (\ref{effective}) for $ |\tau|\leq \theta$.
\end{lemma}

To prove the lemma we first perform some additional constructions.

Assume for definiteness that $\tau\ge 0$, i.e. $0\le\tau\le \theta$,
and consider an intermediate scale $L=\sqrt{\epsilon}$; then $\epsilon\ll L\ll 1$.
Denote $N=[\frac{\theta}{L}]$. Let $b_j=jL,0\leq j\leq N$, $b_{N+1}=\theta$ and $\triangle_j=[b_{j-1},b_j], 1\leq j\leq N+1$. Then $|\triangle_1|=\cdots=|\triangle_N|=L$, $0\leq |\triangle_{N+1}| < L$.

Let the curves $y_{\epsilon}^t(a), t\in\mathbb{R}$, be defined as in (\ref{20}) with $ a= a^{\epsilon}$.
\begin{lemma}\label{h}
For any $0\leq  |\tau| \leq \theta$,\ denote
$\
  I(\tau)=\int_0^{\tau}G(a^{\epsilon}(s),s\epsilon^{-1})\text{d}s,
$
where $G(a^{\epsilon}(s),s\epsilon^{-1})=y^{s\epsilon^{-1}}(a^{\epsilon}(s))-\langle\langle P\rangle\rangle(a^{\epsilon}(s))$. Then uniformly in $\tau\in[0,\theta]$ we have $|I(\tau)|\leq \kappa
(\epsilon)$, where $\kappa(\epsilon)\rightarrow 0$ as $\epsilon\rightarrow 0$, does not depend on $v_0$ if $|v_0|\le R$.
\end{lemma}

\begin{proof}
Denote
\begin{equation*}
  I_j(\tau)=\int_{\triangle_j}G(a^{\epsilon}(s),s\epsilon^{-1})\text{d}s,\ \ 1\leq j\leq N+1.
\end{equation*}
Then $|I|\leq\sum_{j=1}^{N+1}|I_j|$. The term $I_{N+1}$ is trivially small.\ Now consider $I_j$ with $1\leq j\leq N$. We have
\begin{align*}
 |I_j|\leq\left|\int_{b_{j-1}}^{b_j} (G(a^{\epsilon}(s),s\epsilon^{-1})-G(a^{\epsilon}(b_{j-1}),s\epsilon^{-1}) )\text{d}s\right |
 +\left|\int_{b_{j-1}}^{b_j} G(a^{\epsilon}(b_{j-1}),s\epsilon^{-1}) \text{d}s \right|
 =:I_j^1+I_j^2.
\end{align*}
Consider the term $I_j^1$. Since $|s-b_{j-1}|\leq L$, then by (\ref{ae}), we have
\begin{equation*}
  |a^{\epsilon}(s)-a^{\epsilon}(b_{j-1})|\leq L\mathcal{X}(2R).
\end{equation*}
As for any $t$ both vector-fields $y^t$ and  $\langle\langle P\rangle\rangle$  belong to $\text{Lip}_{\chi}(\C^n, \C^n)$, then
\begin{equation*}
  I_j^1\leq 2L\mathcal{X}(2R) \cdot L\mathcal{X}(2R) =2L^2(\mathcal{X}(2R))^{2}.
\end{equation*}
Now consider the term $I_j^2$. We have
\begin{align*}
  {I_j^2}&=\left|\int_{b_{j-1}}^{b_j}y^{s\epsilon^{-1}}(a^{\epsilon}(b_{j-1}))\text{d}s-\int_{b_{j-1}}^{b_j}\langle\langle P\rangle\rangle(a^{\epsilon}(b_{j-1}))\text{d}s\right|\\
  &=\left|\int_{b_{j-1}}^{b_{j-1}+L}\Phi_{\Lambda\epsilon^{-1}\tau}\circ P(\Phi_{-\Lambda\epsilon^{-1}\tau}a^{\epsilon}(b_{j-1}))\text{d}\tau-L\langle\langle P\rangle\rangle(a^{\epsilon}(b_{j-1}))\right|\\
  &=\left|\int_0^L \Phi_{\Lambda\epsilon^{-1}b_{j-1}} \Phi_{\Lambda\epsilon^{-1}\tilde{\tau}} \circ P(\Phi_{-\Lambda\epsilon^{-1}\tilde{\tau}}\Phi_{-\Lambda\epsilon^{-1}b_{j-1}}a^{\epsilon}(b_{j-1}))\text{d}\tilde{\tau} -L\langle\langle P\rangle\rangle(a^{\epsilon}(b_{j-1}))\right|\\
  &=\left|\Phi_{\Lambda\epsilon^{-1}b_{j-1}}\int_0^L\Phi_{\Lambda\epsilon^{-1}\tilde{\tau}} \circ P(\Phi_{-\Lambda\epsilon^{-1}\tilde{\tau}}z)\text{d}\tilde{\tau} -L\langle\langle P\rangle\rangle(a^{\epsilon}(b_{j-1}))\right|,
\end{align*}
where $z:=\Phi_{-\Lambda\epsilon^{-1}b_{j-1}}a^{\epsilon}(b_{j-1})\in \bar B_{2R}$.
Making in the last inequality the substitution $\epsilon^{-1}\tilde{\tau}=t$ and noting that $d\tilde\tau =\epsilon dt= L^2\,dt$, we obtain:
\begin{align*}
I_j^2 &=\left|\Phi_{\Lambda\epsilon^{-1}b_{j-1}} \frac{L}{L^{-1}}\int_0^{L^{-1}}\Phi_{\Lambda t}\circ P(\Phi_{-\Lambda t}z)\text{d}t
  -L\langle\langle P\rangle\rangle(a^{\epsilon}(b_{j-1}))\right|\\
  &=\left|L\Phi_{\Lambda\epsilon^{-1}b_{j-1}}\langle\langle P\rangle\rangle^{L^{-1}}(z)-L\langle\langle P\rangle\rangle(a^{\epsilon}(b_{j-1}))\right|.
\end{align*}
From \eqref{fromlemma}, $\langle\langle P\rangle\rangle^{L^{-1}}(z)=\langle\langle P\rangle\rangle(z)+o(1)$ as $\epsilon\rightarrow 0$. Since by item 3) of Proposition \ref{p}
\[
\langle\langle  P\rangle\rangle(z)= \langle\langle  P\rangle\rangle(\Phi_{-\Lambda \epsilon^{-1}b_{j-1}}a^{\epsilon}(b_{j-1}))= \Phi_{-\Lambda \epsilon^{-1}b_{j-1}}\langle\langle  P\rangle\rangle(a^{\epsilon}(b_{j-1})),
\]
and as by Lemma \ref{a4} the $o(1)$ above does not depend on $ z\in \bar B_{2R}$,  then
\begin{equation*}
   I_j^2=\left|L\langle\langle P\rangle\rangle(a^{\epsilon}(b_{j-1}))+L\cdot o(1)-L\langle\langle P\rangle\rangle(a^{\epsilon}(b_{j-1}))\right| =L\cdot o(1).
\end{equation*}
So
\begin{equation*}
  |I_j|\leq L\cdot o(1)+2L^2(\mathcal{X}(2R))^{2}=L\cdot o(1),
\end{equation*}
and therefore,
\begin{equation*}
  |I|\leq \sum_{j=1}^{N}|I_j|\leq NL\cdot o(1)\leq \theta\cdot o(1)=:\kappa(\epsilon).
\end{equation*}
The lemma is proved.
\end{proof}

\begin{proof}[Proof of Lemma \ref{solutiona0}]
Consider
\begin{align}
\nonumber  A(\tau):=&\ a^0(\tau)-v_0-\int_0^{\tau}\langle\langle P\rangle\rangle(a^0(s))\text{d}s\\
\label{term1}  =&\ a^0(\tau)-a^{\epsilon_j}(\tau)\\
 \label{term2}  &\ \ +a^{\epsilon_j}(\tau)-v_0-\int_0^{\tau}y(a^{\epsilon_j}(\tau),s\epsilon_j^{-1})\text{d}s\\
 \label{term3}  &\ \  +\int_0^{\tau}y(a^{\epsilon_j}(\tau),s\epsilon_j^{-1})\text{d}s-\int_0^{\tau}\langle\langle P\rangle\rangle(a^{\epsilon_j}(s))\text{d}s\\
\label{term4}   &\ \  +\int_0^{\tau}\langle\langle P\rangle\rangle(a^{\epsilon_j}(s))\text{d}s-\int_0^{\tau}\langle\langle P\rangle\rangle(a^0(s))\text{d}s.
\end{align}
The term (\ref{term1})$\,\rightarrow 0$ as $\epsilon_j\rightarrow 0$ in view of (\ref{compact}), the term (\ref{term2})=0 by  (\ref{aet}), the term (\ref{term3})$\leq\kappa(\epsilon)$ by Lemma \ref{h} and (\ref{term4})$\leq\tau\mathcal{X}(2R)|a^{\epsilon_j}-a^0|\rightarrow 0$ as $\epsilon_j\rightarrow 0$ by (\ref{compact}).
Passing to the limit as $\epsilon_j\rightarrow 0$, we see that $A(\tau)\equiv 0$. Therefore, $a^0(\tau)$ is a solution of  (\ref{effective}) for $0\leq {\color{red}|\tau|}\leq \theta$.
Lemma~\ref{solutiona0} is proved.
\end{proof}

Since a solution of eq. (\ref{aet}) is unique, then the convergence (\ref{compact}) holds as $\epsilon\rightarrow 0$, not only as $\epsilon_j\rightarrow 0$. Thus we have proved

\begin{theorem}\label{limit}
Let $a^{\epsilon}(\tau),|\tau|\leq \theta$, be a solution of (\ref{aet}). Then
\begin{equation*}
  a^{\epsilon}(\tau)\rightarrow a^0(\tau),\ \ \ \text{ uniformly  for}\  |\tau|\leq \theta,
\end{equation*}
where $a^0(\tau)$ is a solution of (\ref{effective}).
\end{theorem}

Theorem~\ref{limit} proves the second assertion of Theorem \ref{tbk}, whose first assertion was already established in Lemma~\ref{a4}.

Since $|v_j^{\epsilon}(\tau)|\equiv |a_j^{\epsilon}(\tau)|$, then Theorem \ref{limit} implies:

\begin{corollary}
The solution $v^\epsilon(t)$ of \eqref{2} satisfies
\begin{equation*}
\sup_{ |t|\leq \epsilon^{-1}\theta}
\big|  |v_j^{\epsilon}(t)| - |a_j^{0}(\epsilon t)|\big| \to 0\quad\text{as}\;\; \epsilon\to0, \ \forall j.\ 
\end{equation*}
\end{corollary}

\begin{example}\label{ex3}
By Examples \ref{ex1} and \ref{ex2},\ the effective equation for the system \eqref{ss1}, written in the complex variable $v$ and the
 slow time $\tau$, is
$$
\frac{\partial}{\partial\tau} v = v^2 \bar v.
$$
In the real variables and the fast time the equation reads
$$
\dot x= \eps x(x^2+y^2), \quad \dot y =\eps y(x^2+y^2).
$$
For $x_0=1, y_0=0$ its solution is $x(t) = (1-2\eps t)^{-1/2}$, $y(t) =0$.  So for $|t|\le \tfrac14 \eps^{-1}$ the solution of
\eqref{ss1} with initial data $(1,0)$ satisfies
$$
\big| x^2(t) +y^2(t)\big|^{1/2} =  (1-2\eps t)^{-1/2}  + o(1) \qquad \text{as $\eps\to0$. }
$$
\end{example}
\medskip

\section{ Hamiltonian equations}\label{s_ham}

Let us provide the space $\mathbb{R}^{2n}\sim \mathbb{C}^n$ with the usual symplectic structure, given
by the form $\omega_2=\sum \text{d}x_j\wedge \text{d}y_j$. Then a real-valued Hamiltonian
\begin{equation*}
H= h_2(z,\bar{z})+\epsilon h(z,\bar{z}),\ \ h_2=-\frac{1}{2}\sum\lambda_j|z_j|^2,   \;\; h\in C^1(\C^n),
\end{equation*}
gives rise to the Hamiltonian system
$$
  \dot{z}_j=-\mathbf{i}\lambda_jz_j+2\mathbf{i}\epsilon\frac{\partial h}{\partial\bar{z}_j},\qquad  1\leq j\leq n,
$$
which we rewrite as
\begin{eqnarray}
\label{hamil3'} \frac{\partial z}{\partial \tau}+\mathbf{i}\epsilon^{-1}\text{diag}(\lambda_j)z=2\mathbf{i}\frac{\partial h}{\partial\bar{z}}:=P(z).
\end{eqnarray}
Assume that $P$ is locally Lipschitz. It means that
$$
h\in C^{1,1}_{ \text{loc}} =\{ u(z) \in
C^1(\mathbb{C}^n): u_z, u_{\bar z} \;\;\text{are locally Lipschitz functions}\}.
$$
Then \eqref{hamil3'}  is a special case of equation (\ref{maineq**}).

\begin{lemma}\label{yjt}
Let $y^t(a)$ be defined as in (\ref{20}). Then
\begin{equation*}
  y_j^t(a)=2\mathbf{i}\frac{\partial }{\partial \bar{a}_j}h(\Phi_{-\Lambda t}a).
\end{equation*}
\end{lemma}

\begin{proof}
Denote $\Phi_{-\Lambda t}a=v$, i.e.,  $v_j=e^{-\mathbf{i}\lambda_j t}a_j$. Then
\begin{align*}
  2\mathbf{i}\frac{\partial h(v(a))}{\partial \bar{a}_j} &=2\mathbf{i}\frac{\partial h(v(a))}{\partial \bar{v}_j}\cdot\frac{\partial \bar{v}_j}{\partial \bar{a}_j } =2\mathbf{i}\frac{\partial h(v(a))}{\partial \bar{v}_j} e^{\mathbf{i}\lambda_j t}\\
  &=2\mathbf{i}e^{\mathbf{i}\lambda_j t} \frac{\partial h(\Phi_{-\Lambda t}a)}{\partial \bar{v}_j}
=e^{\mathbf{i}\lambda_j t} P_j(\Phi_{-\Lambda t}a)= y_j^t(a),
\end{align*}
since $2\mathbf{i}\frac{\partial h}{\partial\bar{v}_j}=P_j(v)$.
\end{proof}

From Lemma \ref{yjt},
\begin{equation}\label{relation}
  \langle\langle P\rangle\rangle_j^{T}(a)=\frac{2\mathbf{i}}{|T|}\int_0^T\frac{\partial }{\partial \bar{a}_j}h(\Phi_{-\Lambda t}a)\text{d}t
  =2\mathbf{i}\frac{\partial }{\partial \bar{a}_j}\langle h\rangle^T(a),
\end{equation}
where we denoted
$\
\langle h\rangle^T(a)=\frac{1}{|T|}\int_0^T h(\Phi_{-\Lambda t}a)\mathrm{d} t.
$

For the same reason as in Section \ref{18*} (see there Lemma~\ref{a4}), since $h$ is a  locally Lipschitz function, then
the limit
\begin{equation}\label{limm}
\lim_{T\to \pm\infty} \langle h\rangle^T(a) =: \langle h\rangle(a)
\end{equation}
exists and is  locally Lipschitz (this limit  is the averaging of the function $h$ in the direction $\Lambda$).
 Repeating the proof of
Proposition~\ref{p}.3) we get that $ \langle h\rangle(a)$ is invariant with respect to the rotations $\Phi_{\Lambda \theta}$:
\begin{eqnarray}\label{invariance}
\langle h\rangle(\Phi_{\Lambda \theta}a) = \langle h\rangle(a )\qquad \forall \theta,\; \forall a.
\end{eqnarray}
\newpage

We claim that $\langle h\rangle$ is a $C^1$-function and 
\begin{equation}\label{ham_eff}
 2 \mathbf{i}  \frac{\partial}{ \partial \bar a_j} \langle h\rangle (a) = \langle \langle P\rangle\rangle_j (a) ,
\end{equation}
for every $j$.  Indeed, by  the second assertion of Lemma~\ref{b1}  and \eqref{relation}
$$
  \langle\langle \bar P\rangle\rangle_j^{T}(a)  =-2 \mathbf{i} \frac{\partial }{\partial {a}_j}\langle h\rangle^T(a).
$$
From here and \eqref{relation} we conclude that $\langle h\rangle^T$ is a $C^1$ function, and for each $j$ the derivative 
 $(\partial/\partial x_j) \langle h\rangle^T$ is a complex linear combination of $  \langle\langle P\rangle\rangle_j^{T}$
 and $  \langle\langle \bar P\rangle\rangle_j^{T}$. So by Lemma~\ref{a3}, for any bounded domain $B$ in $\C^n$ 
 functions $(\partial/\partial x_j) \langle h\rangle^T$ are uniformly in $T>0$ bounded and Lipschitz on $B$. Same is true for
 $(\partial/\partial y_j) \langle h\rangle^T$, for all $j$. Now the uniform in $B$ convergence 
  of $ \langle\langle P\rangle\rangle_j^{T}$ to $ \langle\langle P\rangle\rangle_j$ implies the uniform convergence of
  $(\partial/\partial x_j) \langle h\rangle^T$ and $(\partial/\partial y_j) \langle h\rangle^T$ to limits, for all $j$. In view
  of the uniform in $B$ convergence of   $ \langle h \rangle^T$ to $ \langle h\rangle$ we get  that  
 $\langle h\rangle$ is  $C^1$-smooth  and satisfies \eqref{ham_eff}.

We have proved

\begin{theorem}\label{t_ham}
If  equation \eqref{2} has the hamiltonian form \eqref{hamil3'}, where the real Hamiltonian $h$ belongs to the space $C^{1,1}_{ \text{loc}} $,
then the effective equation \eqref{effective} also is hamiltonian: the vector field $\langle \langle P\rangle\rangle$ has the hamiltonian form 
 \eqref{ham_eff}, where the Hamiltonian $ \langle h\rangle \in  C^{1,1}_{ \text{loc}}$ is defined in \eqref{limm}. 
\end{theorem}

\begin{example}
If  $h$ is a real polynomial,
\begin{equation}\label{real_polyn}
  h(a)=\sum_{\substack{\alpha,\beta\in\mathbb{Z}_+^n}} m_{\alpha\beta}a^{\alpha}\bar{a}^{\beta}, \quad m_{\alpha\beta}=\overline{m_{\alpha\beta}},
\end{equation}
then
$
 \langle h\rangle(a) = \sum_{\Lambda\cdot\alpha=\Lambda\cdot\beta}m_{\alpha\beta}a^{\alpha}\bar{a}^{\beta}.
$
\end{example}

Let us take any $h \in  C^{1,1}_{ \text{loc}} $ and
assume that the vector $\Lambda$ is non-resonant, that is $\Lambda\cdot s=0$ for some $s\in \Z^n$ implies that $s=0$. Let us
 introduce in $\C^n$ the
action-angle coordinates
$
(I,\vp)$ with $ I\in \R_+^n,\; \vp\in \T^n,
$
where for a vector
$
z=(z_1, \dots,z_n)$ with $z_j=  r_je^{ \mathbf{i}\vp_j}
$
we have
$
I=(I_1, \dots, I_n), I_j= \tfrac12 r_j^2,
$
and $\vp=(\vp_1, \dots, \vp_n)$. Then $\omega_2= \text{d}I\wedge \text{d}\vp$, and
$$
\Phi_{\Lambda t} (I, \vp) = (I, \vp+\Lambda t).
$$
Since now the curve $t\mapsto \vp+\Lambda t$ is dense in $\T^n$ for each $\vp$, then \eqref{invariance} implies that
$ \langle h\rangle$ does not depend on the angles $\vp$.\footnote{For example, if $h$ is the polynomial \eqref{real_polyn}, then
$
 \langle h\rangle(a) = \sum m_{\alpha\alpha}a^{\alpha}\bar{a}^{\alpha}.
$
}
That is, the Hamiltonian $ \langle h\rangle$ is integrable, and
the effective equation reads
$$
 \dot I=0,\quad \dot\vp= \nabla_I  \langle h\rangle(I).
$$
So its solutions $a(\tau)$ are such that $|a_j(\tau)|^2=\,$const, and Theorem \ref{limit} implies

\begin{corollary}\label{nonresonant}
If the frequency vector
$\Lambda$ is non-resonant and $h\in C^{1,1}_{loc} $,  then a solution $z(t)$  of  equation \eqref{hamil3'} satisfies
\begin{equation*}
  \sup_{| t| \leq \epsilon^{-1}\theta}\left||z_j(t)|^2-|z_{j0}|^2\right
  |=o(1)\ \ \text{as } \epsilon\rightarrow 0,
\end{equation*}
for a suitable $\theta$ which depends on $|z_0|$.
\end{corollary}

If  $h\in C^{1,1}_{ \text{loc}} $, then $2\mathbf{i} \partial_{\bar z} h \in  \text{Lip}_{\mathcal{X}}(\mathbb{C}^{n},\mathbb{C}^{n})$
for a  function $\mathcal{X}$ as in Definition~\ref{8}. For $ m\in \N$ we  write
$$
\hh(z) = O(z^m),
$$
if
$
\epsilon^{-m} \hh (\epsilon z) \in \text{Lip}_{\mathcal{X}}(\mathbb{C}^{n},\mathbb{C}^{n})$ for all $ 0<\epsilon\le1,
$
for a suitable function $\mathcal{X}$, independent from $\epsilon$. 
 For example, if $h$ is a polynomial \eqref{real_polyn} such that the coefficients
$m_{\alpha \beta}$ are nonzero only for $|\alpha|+|\beta| \ge M$, then $\hh =O(z^{M-1})$. Consider the equation
\begin{equation}\label{small_ham}
 \dot z  +\mathbf{i}\, \text{diag}(\lambda_j)z=\hh,
\end{equation}
where $\hh \in O(z^m)$. To study its small solutions we substitute $ z=\epsilon w$ and get for $w(t)$
 equation \eqref{hamil3'} with $\epsilon := \epsilon^{m-1}$. We see that if the frequency vector $\Lambda$ is non-resonant and $z(t)$
is a solution of \eqref{small_ham} with small initial data $z(0) = \epsilon w_0$, then
$
\big| |z_j (t)|^2 - \epsilon^2 |w_{0j}|^2\big| =  o(\epsilon^2)
$
for $|t|\le \epsilon^{1-m} \theta$, where $\theta=\theta(|w_0|)$. In \cite{Gior} a  more delicate argument is used to show that
if the frequency vector $\Lambda$
 satisfies certain diophantine condition and the Hamiltonian $h$ is analytic, then the stability interval
is much bigger -- it is exponentially long in terms of $\epsilon^{-a}$ for some positive $a$. Our result is significantly weaker,
but it only requires that the vector $\Lambda$ is non-resonant and the Hamiltonian vector-field $\hh$ is Lipschitz--continuous.
We note that the result of  \cite{Gior} generalises to PDEs, e.g. see \cite{Bam}, as well as  Theorem~\ref{tbk}
(and Corollary~\ref{nonresonant}), see \cite{H-G1, HKM}. 

\footnotesize
\bibliographystyle{abbrv} 
\bibliography{averaging}

\begin{thebibliography}{10}

\bibitem{Arn}
V.~I. Arnold.
\newblock {\em Geometrical Methods in the Theory of Ordinary Differential
  Equations}.
\newblock Springer-Verlag, New York, second edition, 1988.

\bibitem{AKN}
V.~I. Arnold, V.~V. Kozlov, and A.~I. Neishtadt.
\newblock {\em Mathematical Aspects of Classical and Celestial Mechanics}.
\newblock Springer-Verlag, New York, third edition, 2006.

\bibitem{Bam}
D.~Bambusi.
\newblock An averaging theorem for quasilinear {H}amiltonian {PDE}s.
\newblock {\em Ann. Henri Poincar\'{e}}, 4(4):685--712, 2003.

\bibitem{BM}
N.~N. Bogoliubov and Y.~A. Mitropolsky.
\newblock {\em Asymptotic Methods in the Theory of Non-linear Oscillations}.
\newblock Gordon and Breach, New York, 1961.

\bibitem{G}
I.~M. Gelfand.
\newblock {\em Lectures on Linear Algebra}.
\newblock Dover Publications, Inc., New York, 1989.

\bibitem{Gior}
A.~Giorgilli, A.~Delshams, E.~Fontich, L.~Galgani, and C.~Sim\'o.
\newblock Effective stability for a {H}amiltonian system near an elliptic
  equilibrium point, with an application to the restricted three-body problem.
\newblock {\em J. Differential Equations}, 77:167--198, 1989.

\bibitem{H-G1}
G.~Huang.
\newblock An averaging theorem for nonlinear {S}chr\"{o}dinger equations with
  small nonlinearities.
\newblock {\em Discrete Contin. Dyn. Syst. A}, 34:3555--3574, 2014.

\bibitem{HKM}
G.~Huang, S.~Kuksin, and A.~Maiocchi.
\newblock Time-averaging for weakly nonlinear {CGL} equations with arbitrary
  potentials.
\newblock In {\em Hamiltonian partial differential equations and applications},
  volume~75 of {\em Fields Inst. Commun.}, pages 323--349. Fields Inst. Res.
  Math. Sci., Toronto, ON, 2015.

\bibitem{Khas}
R.~Khasminski.
\newblock On the avaraging principle for {I}to stochastic differential
  equations.
\newblock {\em Kybernetika}, 4:260--279, 1968.
\newblock (in Russian).

\bibitem{R}
W.~Rudin.
\newblock {\em Real and Complex Analysis}.
\newblock McGraw-Hill, New York, 1987.

\end{thebibliography}

\end{document}